\documentclass[]{amsart}

\usepackage{graphicx, xcolor}
\usepackage{amssymb}
\usepackage{subfigure}
\usepackage{caption}
\usepackage{amsfonts}
\usepackage{amsmath,amsthm}
\usepackage{enumerate}
\usepackage{epstopdf}

\newcommand{\psl}{\psi_{\lambda}}
\newcommand{\supp}{\text{supp }}
\newcommand{\Lpd}{L^2(\mathbb{R}^d)}
\newcommand{\bn}{\mathbf{n}}
\newcommand{\bs}{\mathbf{s}}
\newcommand{\mm}{\mathcal{M}}
\theoremstyle{plain}
\newtheorem{theorem}{Theorem}
\newtheorem{prop}{Proposition}
\newtheorem{definition}{Definition}

\theoremstyle{definition}
\newtheorem{remark}{Remark}
\newtheorem{problem}{Problem}

\title[Reconstruction from unsigned frame coefficients]{Reconstructing real-valued
functions from unsigned coefficients with respect to wavelet and other
frames}
\author{Rima Alaifari}
\address[Rima Alaifari]{ Department of Mathematics, ETH Z\"{u}rich, R\"{a}mistrasse 101, 8092 Z\"{u}rich, Switzerland}
\author{Ingrid Daubechies}
\address[Ingrid Daubechies]{Department of Mathematics, Duke University, Durham/NC 27708-0320, USA}
\author{Philipp Grohs}
\address[Philipp Grohs]{ Department of Mathematics, ETH Z\"{u}rich, R\"{a}mistrasse 101, 8092 Z\"{u}rich, Switzerland and
Faculty of Mathematics, University of Vienna, Oskar-Morgenstern-Platz 1, 1090 Wien, Austria}
\author{Gaurav Thakur}
\address[Gaurav Thakur]{INTECH Investment Management, One Palmer Square 441, Princeton, NJ 08542}



\begin{document}

\maketitle

\begin{abstract}
In this paper we consider the following problem of phase retrieval: Given a collection of real-valued band-limited functions $\{\psi_{\lambda}\}_{\lambda\in \Lambda}\subset L^2(\mathbb{R}^d)$ that constitutes a semi-discrete frame, we ask whether any real-valued function $f \in L^2(\mathbb{R}^d)$ can be uniquely recovered from its unsigned convolutions ${\{|f \ast \psi_\lambda|\}_{\lambda \in \Lambda}}$. 

We find that under some mild assumptions on the semi-discrete frame and if $f$ has exponential decay
at $\infty$, it suffices to know $|f \ast \psi_\lambda|$ on suitably fine lattices to uniquely 
determine $f$ (up to a global sign factor). 

We further establish a local stability property of our reconstruction problem. Finally, for two concrete examples of a (discrete) frame of $L^2(\mathbb{R}^d)$, $d=1,2$, we show that through sufficient oversampling one obtains a frame such that any real-valued function with exponential decay can be uniquely recovered 
from its unsigned frame coefficients.
%
\end{abstract}

\section{Introduction}

In phase retrieval, one wishes to recover a function from phaseless (or unsigned) measurements. A classical example is the problem of reconstructing a function $f$ from its Fourier modulus $|\hat{f}|$ \cite{hurt2001phase, klibanov1995phase}, a question which appears in e.g. X-ray crystallography. It is well-known that this problem is not uniquely solvable \cite{akutowicz1956, walther1963}.

In the discrete setting, phase retrieval is typically formulated as recovering a finite signal $x \in \mathbb{C}^N$ from $\{ |\langle x,y_i \rangle |^2\}_{i=1}^m$, where $\{y_i\}_{i=1}^m\subset \mathbb{C}^N$ is a collection of measurement vectors (such as elements of a frame for audio processing applications or complex exponentials in the Fourier modulus case) \cite{balan2006signal, Bandeira_SavingPhase, candes2015phase}. 

An interesting setting concerning phase retrieval that combines the continuous framework and frame aspects, studied in \cite{mw}, seeks to recover a function ${f\in L^2(\mathbb{R})}$ from the magnitudes  of its semicontinuous wavelet coefficients, i.e. from $\{|f\ast \psi_j |\}_{j\in \mathbb{Z}}\subset L^2(\mathbb{R})$, where $\psi_j(x):=2^{j}\psi(2^jx)$. 
There it is shown that if $\psi$ is a Cauchy wavelet, that is, $\hat \psi (\xi)=\xi^p \exp(-\xi)\chi_{\xi>0}$, then any $f\in L^2(\mathbb{R})$ can be uniquely determined from $\{|f\ast \psi_j |\}_{j\in \mathbb{Z}}$, up to a global phase.

In this paper we consider a setting complementary to the one in \cite{mw}. Given a collection  $\{\psi_{\lambda}\}_{\lambda\in \Lambda}\subset L^2(\mathbb{R}^d)$ of \emph{real-valued and band-limited} functions  such that the system $\Psi_\Lambda:=\{\psi_\lambda(\cdot - u)\}_{u\in \mathbb{R}^d,\lambda\in \Lambda}$
constitutes a semi-discrete frame as defined in Section \ref{sec:probform} below, 
we study the recoverability of real-valued, multivariate $f\in L^2(\mathbb{R}^d)$ from 
unsigned convolutions $\{|f\ast \psi_\lambda|\}_{\lambda\in \Lambda}$.

Our main result states that, under very general conditions on the semi-discrete frame $\Psi_\Lambda$ (precisely stated in Section \ref{sec:main}), any real-valued function ${f\in L^2(\mathbb{R}^d)}$ that decays exponentially at $\infty$ can be uniquely
recovered, up to a global sign, from unsigned samples of the convolutions $|f\ast \psi_\lambda|$ on sufficiently fine lattices. We formulate our results for exponentially decaying functions for simplicity, but this condition can be further relaxed, as can be seen from the proof of Theorem \ref{mainprop}. The main requirement can be understood as a condition on the zero set of the Fourier transform $\widehat{f}$. In finite-dimensions, controlling the zero set of the measurements has been used to study uniqueness of some classes of phase retrieval problems \cite{alexeev2014phase,bodmann2016algorithms}. A recent paper studies phase retrieval in the real-valued setting for signals in shift-invariant spaces \cite{chen2016phase}.

The arguments exploited in our approach are similar to those in previous works \cite{alaifari-grohs} and \cite{thakur} by the authors, on the recovery of univariate band-limited and 
real-valued functions $f\in L^2(\mathbb{R})$ from unsigned samples. For the method of \cite{thakur} it was essential that $f$ be real-valued; our approach inherits the same restriction. Consequently, in contrast to the results of \cite{mw}, our result works only for real-valued and band-limited semi-discrete frames. Aside from that, however, the frame $\Psi_\Lambda$ can be rather arbitrary. In particular, our results cover general multivariate wavelet frames, curvelet frames, ridgelet frames and many others. 

Another difference with \cite{mw} is that we only require the unsigned convolutions to be known on a sufficiently dense sampling lattice. For instance, if $\psi$ is a real-valued and band-limited wavelet we show that, for some oversampling constant $\alpha>0$ (which is determined explicitly by the bandwidth of $\psi$) the data $\{| \langle f , 2^{j/2}\psi(2^j\cdot - \alpha k)\rangle|\}_{j,k\in\mathbb{Z}}$ uniquely determine any  $f\in L^2(\mathbb{R})$ that has exponential decay at $\infty$; similar sampling results hold in the multivariate setting and for more general frames such as curvelets, shearlets or ridgelets. 

The two different settings, namely the one considered here and the one of \cite{mw}, might give insight into the general problem of recovering a function
from phaseless information.

This paper is organized as follows. In Section \ref{sec:probform} we first recall the definition of a semi-discrete frame and then state the precise formulation of the problem. In Section \ref{sec:approach} we derive a multidimensional generalization of the result in \cite{thakur}. Next, this is used to show in Section \ref{sec:main} our main result (Theorem \ref{mainprop}) on recovering a function $f$ from unsigned samples of $|f\ast \psi_\lambda|$. The stability properties of our reconstruction problem are the subject of Section \ref{sec:stability}, where we derive a local stability result (Theorem \ref{thm-stab}). In Section \ref{sec:examples}  we consider two concrete examples of discrete frames: an orthonormal basis of Meyer wavelets for $L^2(\mathbb{R})$ and a tight frame of curvelets for $L^2(\mathbb{R}^2)$. We show in Theorem \ref{phr-meyer}  that  suitably oversampling the Meyer wavelet basis yields a frame such that any exponentially-decaying, real-valued function in $L^2(\mathbb{R})$ can be uniquely recovered from its unsigned frame coefficients. Similarly, Theorem \ref{phr-curvelets} gives a result in $2D$: Sufficient oversampling of the curvelet frame results in a frame for which any exponentially-decaying, real-valued function can be uniquely recovered from its unsigned frame coefficients.

We fix the following notations: For a set $C \subseteq \mathbb{R}^d$, we use the notation $L^2(C,\mathbb{R}):=\{ f \in L^2(C); \ f:C \to \mathbb{R}\}$. Similarly, for a discrete set $\Lambda \subset \mathbb{R}^d$, we denote the $\ell^2-$space with codomain $Y$ by $\ell^2(\Lambda,Y)$. 
The range of an operator $T$ will be written as $\mathcal{R}(T)$. Convolution is denoted by $\ast$, complex conjugation of $f$ by $\overline{f}$ and the Fourier transform is normalized as $\hat{f}(\xi) = \int_{\mathbb{R}^d} f(x) e^{-2 \pi i x \cdot \xi} dx;$ with this notation, 
$\widehat{(f\ast g)}(\xi)=\hat{f}(\xi)\hat{g}(\xi)$.
To describe a true subset we use the symbol $\subset$ 
(i.e. 
$E \subset F$ implies $E \neq F$; if $E=F$ is allowed, we write 
$E\subseteq F $) and for $k \in \mathbb{R}$, $k\mathbb{Z} = \{ k \cdot n; \ n \in \mathbb{Z} \}$. If $M$ is a $d \times d-$dimensional real matrix, then $M^T$ denotes its transpose.

\section{Problem formulation}\label{sec:probform}
Let $\Lambda$ be a discrete set of indices and $\{ \psi_\lambda \}_{\lambda \in \Lambda}$ a collection of real-valued functions in $L^2(\mathbb{R}^d)$ for some dimension $d \in \mathbb{N}$. We assume that each $\psi_\lambda$ is band-limited, i.e. that each $F_\lambda := \supp (\hat{\psi}_\lambda) \subset \mathbb{R}^d$ is compact. We further define ${\Psi_\Lambda:=\{\psi_\lambda(\cdot - u)\}_{u\in \mathbb{R}^d,\lambda\in \Lambda}}$ and require that $\Psi_\Lambda$ constitutes a so-called \textit{semi-discrete} frame. More precisely,
\begin{equation}\label{frame}
A \| f \|_2^2 \leq \sum_{\lambda \in \Lambda} \| f \ast \psl \|_2^2 \leq B  \| f \|_2^2
\end{equation}
for all $f \in L^2(\mathbb{R}^d)$, or equivalently,
\begin{equation}\label{pu}
A \leq \sum_{\lambda \in \Lambda} |\widehat{\psi}_\lambda(\xi)|^2 \leq B
\end{equation}
for some frame bounds $B \geq A > 0$. Note that this is not a frame in the classical sense: in a semi-discrete frame the translation parameter $u$ is left unsampled (i.e. one considers $u \in \mathbb{R}^d$ rather than restricting it to a discrete subset of $\mathbb{R}^d$). Let $\widetilde{\psi}_\lambda$ denote the \textit{dual} frame elements 
$$\widehat{\widetilde{\psi}}_\lambda(\xi) = \frac{\widehat{\psi}_\lambda(\xi)}{\sum_{\lambda \in \Lambda} |\widehat{\psi}_\lambda(\xi)|^2}.$$
Then, any function $f \in L^2(\mathbb{R}^d)$ can be reconstructed from $\{ f \ast \psl \}_{\lambda \in \Lambda}$ as
\begin{equation*}
f(x) = \sum_{\lambda \in \Lambda} \big( (f \ast \psl) \ast \overline{\widetilde{\psi}_\lambda} \big) (x).
\end{equation*}
With these assumptions the problem of phase retrieval considered here can be formulated as follows:

\begin{problem}\label{prob}
Let $f \in \Lpd$ be a real-valued function and assume that $f$ has exponential decay at $\infty$. We define $f_{\lambda} = f \ast \psi_\lambda$ and consider $d-$dimensional lattices $X_{\lambda} = \{x_{\lambda,\bn}\}_{\bn \in \mathbb{Z}^d}$ defined for each $\lambda \in \Lambda$. Given the set of discrete sample magnitudes $\{ |f_\lambda|(X_\lambda)\}_{\lambda \in \Lambda} = \{ |f_\lambda(x_{\lambda,\bn})|\}_{\lambda \in \Lambda, \bn \in \mathbb{Z}^d}$,  we ask whether it is possible to recover $f$ uniquely (up to a global sign). 
\end{problem} 
Clearly, the answer will depend on the sampling scheme for each scale $\lambda$. In addition, some mild additional assumptions on the family $\{ \psi_\lambda \}_{\lambda \in \Lambda}$ will be needed.

We note that because $f$ has exponential decay, i.e because $\int_{\mathbb{R}^d}|f(x)| e^{\alpha |x|} dx < \infty$ for some $\alpha >0$, its Fourier transform extends to an analytic function on 
a ``strip'' $\{z \in \mathbb{C}^d; |\mbox{Im } z_j| < \alpha/d , j=1,\ldots,d\}$ and is real analytic on $\mathbb{R}^d$. 

This implies that knowing $ f \ast \psi_\lambda$ for even one $\lambda$ already determines $f$ (within the class of functions with the same exponential decay); in principle knowing sufficiently fine-grained samples of one $| f \ast \psl|$ would therefore similarly determine $f$ uniquely. However, this approach would be extremely unstable and totally unfeasible in practice: even within the class of compactly supported functions, one can, for all $\epsilon > 0$, find $\tilde{f}_\epsilon$ such that both $\| (f-\tilde{f}_\epsilon) \ast \psl \|_{L^2(\mathbb{R}^d)} \leq \epsilon$ and $\| f-\tilde{f}_\epsilon \|_{L^\infty(C)} \leq \epsilon$, where $C = \supp f$, even though $\| f-\tilde{f}_\epsilon \|_{L^2(\mathbb{R}^d)} \geq 1$. With our approach, requiring a covering condition for the $F_\lambda = \supp \widehat{\psi}_\lambda$ (see below) and using samples $|f \ast \psl|$ for all $\lambda \in \Lambda$, this type of instability is avoided.

\section{The approach}\label{sec:approach}

In earlier work by one of the authors \cite{thakur}, it was shown that if a real-valued band-limited function $g \in L^2(\mathbb{R})$ can be reconstructed from the absolute values of samples $|g(y_k)|$ of $g$ if the sampling frequency is at least twice the Nyquist frequency of $g$. For $d=1$, an immediate consequence is that for Problem \ref{prob}, each $f_\lambda$ can be reconstructed up to a global sign, if sufficiently finely sampled values of $|f_{\lambda}|$ are given. This follows from the assumptions that $f$ and $\psl$ are real-valued and that $\psl$ is band-limited for each $\lambda \in \Lambda$. It then remains to piece together $f$ from the $\sigma_\lambda (f \ast \psi_\lambda)$, where the $\sigma_\lambda \in \{-1,1\}$ are arbitrary and unknown. \\

For $d>1$, we instead use a non-constructive approach similar to \cite[Theorem 3.5]{alaifari-grohs} to prove a multivariate version of \cite[Theorem 1]{thakur}.
\begin{prop}\label{thakur1-multid-simple}
Let $g \in L^2(\mathbb{R}^d)$ be real-valued and band-limited, with
$\supp \hat{g} \subseteq \left[-\frac{1}{2},\frac{1}{2} \right]^d$ 
and set $X= D[2 \bs]^{-1} \mathbb{Z}^d$, where 
$D[\bs] = \text{diag}(s_1,\dots,s_d)$, with $s_1, \ldots, s_d \geq 1$. 
Then, up to a global sign, $g$ can be uniquely recovered from the samples $|g|(X)$.
\end{prop}
\begin{proof}
We argue by contradiction. Suppose that $g,h$ both satisfy the requirements of the statement and that
$|g(X)|=|h(X)|$, but $g\neq \pm h$.
Let 
$X_1:=\{x\in X:\ \text{sgn }g(x)=\text{sgn }h(x)\}$ and $X_2:=X\setminus X_1$. Put $u:=g-h$ and $v:=g+h$. By our assumption that $g\neq \pm h$, we know
that $u\neq 0$ and $v\neq 0$. Furthermore, we have that $u(X_1)=0$ and $v(X_2)=0$. The function $w:= u\cdot v$ hence satisfies that $w(X)=0$. 
Furthermore, the function $w$ is bandlimited with $\supp \hat{w} \subseteq [-1,1]^d$ which, by the choice of the sampling set $X$ implies that $w=0$.
Since $u$ and $v$ are both holomorphic functions, this implies that either $u=0$ or $v=0$ which is a contradiction.
\end{proof}

For $d=1$, it was shown in \cite{thakur} and \cite[p. 18]{thakur2011three} that the oversampling factor of $2$ times the Nyquist frequency is sharp, and for any $s<1$ there exist bandlimited functions $h_j$, $j \in \{1,2\}$, with $\mathrm{supp}(\hat{h_j}) \subseteq \left[-\frac{1}{2},\frac{1}{2} \right]$ such that $|h_1((2s)^{-1}\mathbb{Z})| = |h_2((2s)^{-1}\mathbb{Z})|$ but $h_1 \neq \pm h_2$. It immediately follows that the oversampling factor of $2^d$ ($2$ in each dimension) in Proposition \ref{thakur1-multid-simple} is also sharp. Let $f$ be any function in $L^2(\mathbb{R}^{d-1})$ with bandwidth $\left[-\frac{1}{2},\frac{1}{2} \right]^{d-1}$. If $s_i<1$ for some $i$, then the functions $g_j(x) = f(x_1, \dots, x_{i-1}, x_{i+1}, \dots, x_d) h_j(x_i)$, $j \in \{1,2\}$, have bandwidth $\left[-\frac{1}{2},\frac{1}{2} \right]^d$ and have the same absolute values on $D[2 \bs]^{-1} \mathbb{Z}^d$ but are not equal up to a sign.\\

The support of $\hat{g}$ need not be aligned with the coordinate axes of $\mathbb{R}^d$. We introduce the following:
\begin{definition}\label{approp-sampling}
Let $F$ be a compact set in $\mathbb{R}^d,$ and $M$ a non-singular matrix on $\mathbb{R}^d,$ such that $F \subseteq M \left[-\frac{1}{2}, \frac{1}{2}\right]^d$. Then a \textit{sign-blind sampling set for $F$} is any lattice of the form
$X = (M^{\top})^{-1}  D[2 \bs]^{-1} \mathbb{Z}^d$, where $D[\bs] = \text{diag}(s_1,\dots,s_d),$ and $s_1, \ldots, s_d \in \mathbb{R}^d_+$ with $s_1, \ldots, s_d \geq 1$.
\end{definition}

\begin{remark}
Note that if a lattice $X$ is a sign-blind sampling set for $F$, that it
is also automatically a set of stable sampling for $F$, i.e. 
that every function
$g \in L^2(\mathbb{R}^d)$ with $\supp \hat{g} \subseteq F$ is completely
and stably determined by the sequence of samples $g(X)$. We also note that Definition \ref{approp-sampling} can be easily generalized to include lattices $X$ shifted by an arbitrary vector $v \in \mathbb{R}^d$, i.e., $$X = (M^{\top})^{-1} D[2\bs ]^{-1}\mathbb{Z}^d + v$$ (see Figure \ref{fig:samplingsets}), but we do not consider this case to simplify the notation.
\end{remark}

\begin{figure}[ht!]
     \begin{center}
      \subfigure{
            \label{fig:suppF}
            \includegraphics[width=0.47\textwidth]{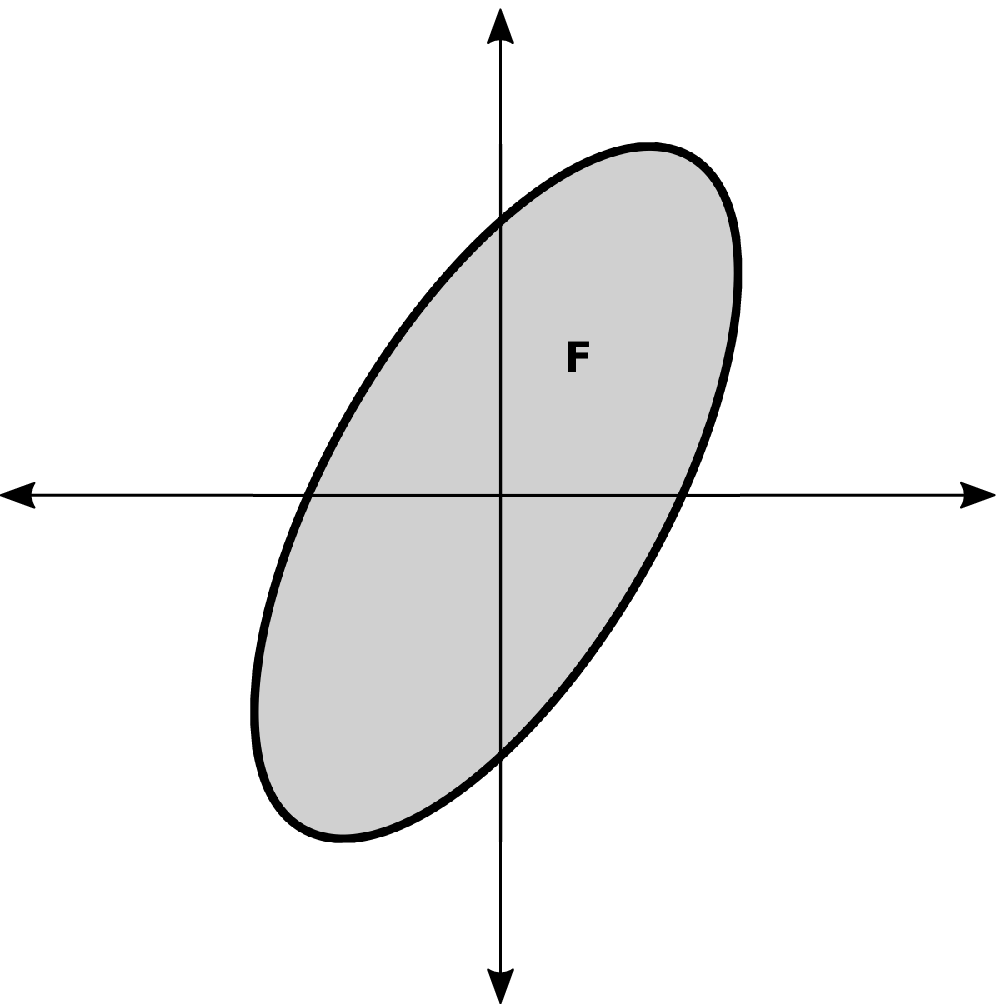}
        }\\
        \subfigure{
            \label{fig:samplingset}
            \includegraphics[width=0.47\textwidth]{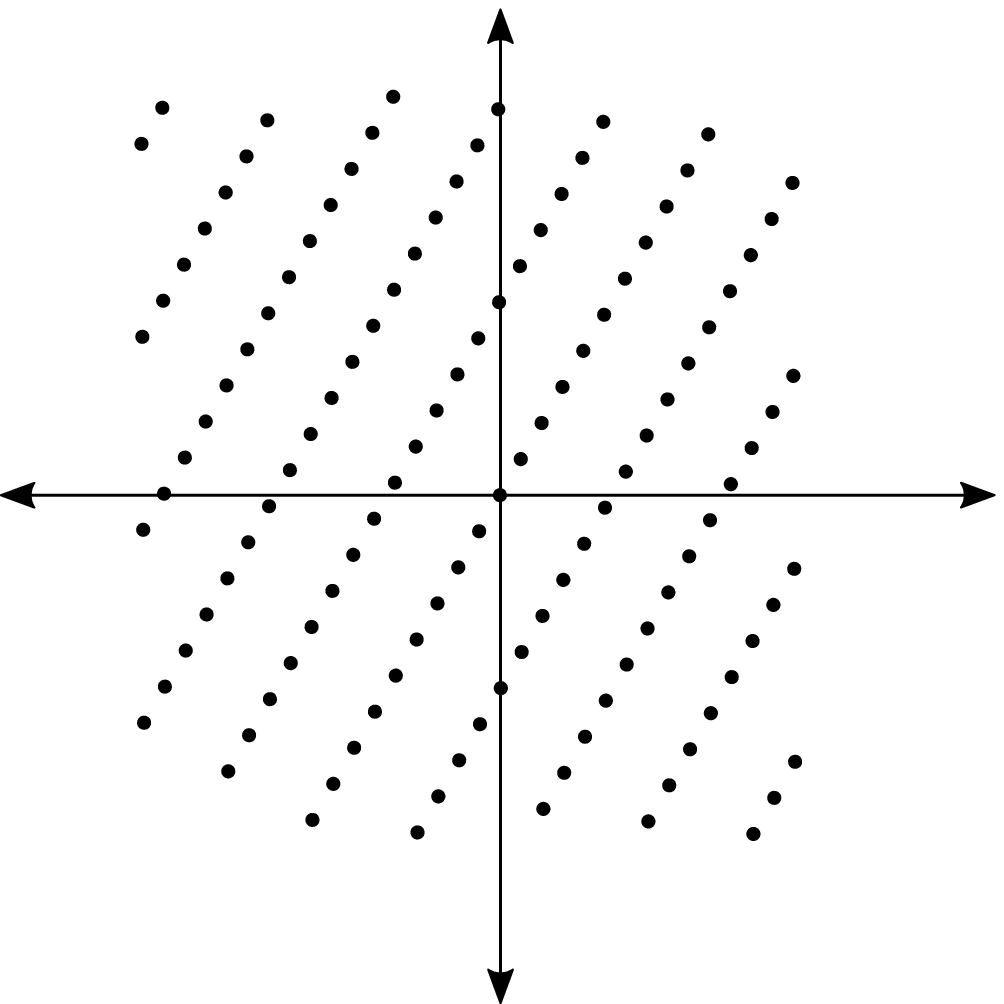}
        }
         \subfigure{
            \label{fig:samplingset-v}
            \includegraphics[width=0.47\textwidth]{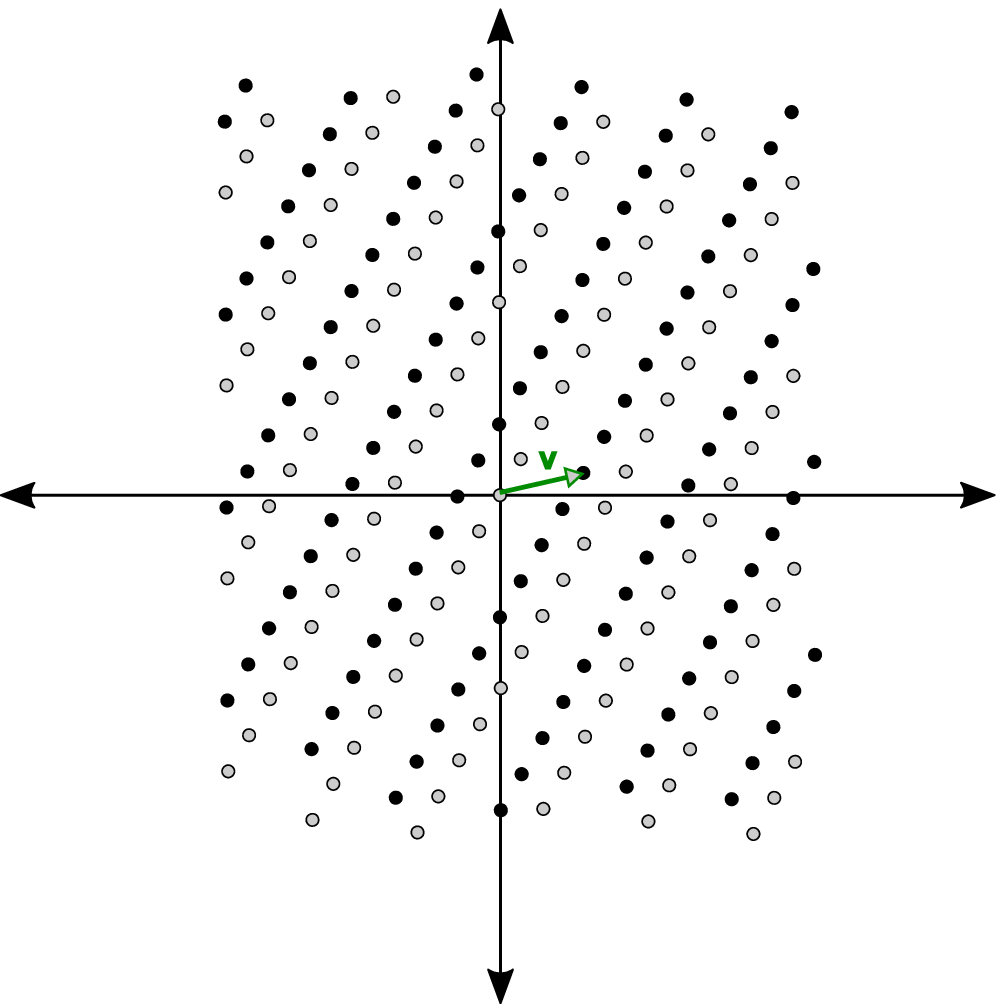}
        }
      \end{center}
    \caption{\textit{Top:} A compact set $F$. \textit{Bottom left:} A sign-blind sampling set $X$ for $F$ centered at the origin. \textit{Bottom right:} The set $X$ (shaded dots) and its arbitrarily shifted version $X+v$ (solidly black dots). 
     }

   \label{fig:samplingsets}
\end{figure}

A slightly more general version of Proposition \ref{thakur1-multid-simple} can now be formulated as follows.
\begin{prop}\label{thakur1-multid}
Let $g \in L^2(\mathbb{R}^d)$ be real-valued and band-limited and let $X$ be a sign-blind sampling set for $\supp \hat{g}$. Then, up to a global sign, $g$ can be uniquely recovered from the samples $|g|(X)$.
\end{prop}
\begin{proof}
The proof follows immediately from Proposition \ref{thakur1-multid-simple}, after a simple change of variables $x \to Mx$.
\end{proof}

Note that when $\supp \hat{g}$ does not align well with the coordinate axes, 
the critical density can depend on the lattice directions determined by the matrix $M$.
More precisely, given a compact set $F$, there can exist
several matrices $M_{\ell}$, with {\em different determinants}, 
such that $F \subseteq M_{\ell} \left[-\frac{1}{2}, \frac{1}{2}\right]^d$ but 
$F \nsubseteq M_{\ell} \left[-\frac{t}{2}, \frac{t}{2}\right]^d$ for $t<1$. 
The corresponding
lattices $X_{\ell}=(M_\ell^{\top})^{-1} D[2\bs ]^{-1} \mathbb{Z}^d$ then have 
densities $2^d \det(M_\ell) \Pi_{j=1}^d s_j$; the requirement that all $s_j \geq 1$ corresponds to 
the different ``critical'' density values $2^d \det(M_\ell)$, for
the same set $F$.

\section{Unsigned samples from a frame}\label{sec:main}

Since $f$ is real-valued and the $\psl$ are real-valued and band-limited in the formulation of Problem \ref{prob}, it follows from Proposition \ref{thakur1-multid} that each $f_\lambda = f \ast \psl$ can be recovered up to an unknown sign factor $\sigma_\lambda$, if $|f_\lambda|$ is sampled at a rate that in each direction is sufficiently high to meet the assumptions of the proposition. In order to allow for recovery of $f$ up to a global sign from the individual $\sigma_\lambda f_\lambda$, an additional criterion on the frame family is needed:
\begin{definition}\label{overlap-supp} 
Let $\Psi_\Lambda$ constitute a semi-discrete frame of real-valued band-limited functions in $L^2(\mathbb{R}^d)$. For each $\lambda \in \Lambda$, let $F_\lambda = \supp \hat{\psi}_\lambda$. Then, we say that the frame has \textit{good Fourier-support overlap} (or \textit{good F-support overlap}) if for all $\lambda \in \Lambda$ there is a non-empty set $U_\lambda \subset F_\lambda$, with $U_\lambda$ open in $\mathbb{R}^d$, such that the $U_\lambda$ cover the whole space, i.e.,
\begin{equation}\label{open-cover}
\bigcup_{\lambda \in \Lambda} U_\lambda = \mathbb{R}^d.
\end{equation}
\end{definition}
\begin{remark}
It is a consequence of this definition that if we define $\widetilde{U}_{\lambda,\mu}$ by 
$$\widetilde{U}_{\lambda,\mu} = U_\lambda \cap U_\mu,$$ 
then, whenever $\widetilde{U}_{\lambda,\mu} \neq \emptyset$, there exists a subregion $\widetilde{F}_{\lambda, \mu} \subseteq \widetilde{U}_{\lambda,\mu}$ on which $|\hat{\psi}_{\lambda}| \geq c_\lambda>0$, $|\hat{\psi}_{\mu}| \geq c_\mu >0$.
\end{remark}

We are now ready to formulate our main result, stating that for a semi-discrete frame with good F-support overlap and corresponding sign-blind sampling sets $\{ X_\lambda\}_{\lambda \in \Lambda} = \{ \{ x_{\lambda, \bn}\}_{\bn \in \mathbb{Z}^d}\}_{\lambda \in \Lambda}$ a function $f$ with
exponential decay is uniquely determined, up to a global sign factor, by the samples $\{ |f \ast \psl| (X_\lambda)\}_{\lambda \in \Lambda}$:
\begin{theorem}\label{mainprop}
Let  $\Psi_\Lambda$ be a semi-discrete frame of real-valued band-limited functions that have good F-support overlap and let $X_\lambda = \{x_{\lambda, \bn}\}_{\bn \in \mathbb{Z}^d} =  (M_\lambda^T)^{-1} D[ 2 \bs_\lambda ]^{-1} \mathbb{Z}^d $ be a sign-blind sampling set for each $F_\lambda=\supp \widehat{\psi}_{\lambda}$. Furthermore, let $f \in L^2(\mathbb{R}^d)$ be real-valued and satisfy
$\int_{\mathbb{R}^d} |f(x)| e^{\alpha|x|}dx < \infty$ for some $\alpha>0$. Then, $f$ can be uniquely recovered up to one global sign factor. In fact, if any $h \in L^2(\mathbb{R}^d,\mathbb{R})$ satisfies $|f \ast \psl|(X_\lambda) = |h \ast \psl|(X_\lambda)$ for all $\lambda \in \Lambda$, then $h = \sigma f$ with an unknown sign factor $\sigma \in \{-1,1\}$.
\end{theorem}
\begin{proof}
We can apply Proposition \ref{thakur1-multid} for any $\lambda \in \Lambda$ to reconstruct $f_\lambda$ and hence, $\widehat{f}_\lambda = \widehat{f} \widehat{\psi}_{\lambda}$, up to a sign factor. We proceed in two steps, the first local, the second knitting these together to get a global argument.

If $\psi_{\lambda}, \psi_{\mu}$ are any two frame elements for which $\widetilde{U}_{\lambda,\mu} \neq \emptyset,$ consider the set  $\widetilde{F}_{\lambda, \mu}$ defined as in the remark below Definition \ref{overlap-supp}. By the analyticity of $\widehat{f}$ on the $d$-dimensional ``strip'' 
$\{z \in \mathbb{C}^d; |\mbox{Im} z_j| <\alpha/d,\,j=1,\ldots d\}$, 
it follows that we can choose a subregion $J_{\lambda,\mu} \subseteq \widetilde{F}_{\lambda, \mu}$, open in $\mathbb{R}^d$, on which $|\hat{f}|>0$.

Now denote by $g_{\lambda}$, $g_{\mu}$ the reconstructions obtained, i.e., 
\begin{align*}
\widehat{g}_{\lambda} &= \sigma_\lambda \widehat{f}_{\lambda} = \sigma_\lambda  \widehat{f} \widehat{\psi}_{\lambda},\\ 
\widehat{g}_{\mu} &= \sigma_\mu \widehat{f}_{\mu} = \sigma_\mu \widehat{f} \widehat{\psi}_{\mu}.
\end{align*}
We can then use that $\widehat{g}_{\lambda}/\widehat{\psi}_{\lambda}$ and $\widehat{g}_{\mu}/\widehat{\psi}_{\mu}$ must coincide on $J_{\lambda,\mu}$ to match the two reconstructions and eliminate one of the two sign factors. 

To assemble these local arguments into a global reconstruction of $f$, we start with an initial compact subset $K_1 \subset \mathbb{R}^d$. By the compactness, there is a cover $\{ U_\lambda: \lambda \in \Lambda_{K_1}\}$ of $K_1$, with $\Lambda_{K_1} \subset \Lambda$ being finite. The signs of $\widehat{g}_\lambda$, $\lambda \in \Lambda_{K_1},$ can be matched by repeating the above step $|\Lambda_{K_1}|-1$ times.

We now proceed iteratively: Let $\{ U_\lambda: \lambda \in \Lambda_{K_i}\}$ be the chosen finite cover of the compact set $K_i$. We define the compact set $K_{i+1}$ as the closure of this cover, i.e., $K_{i+1} = \overline{\bigcup_{\lambda \in \Lambda_{K_1}} U_\lambda}$. By definition, $K_{i+1} \supsetneq K_i$. For the finite covering of $K_{i+1}$ we choose a collection of $U_\lambda$'s such that the covers are nested, i.e., $\{U_\lambda: \lambda \in K_{i+1}\} \supset \{ U_\lambda: \lambda \in K_{i}\}.$ 
The signs of $\widehat{g}_\lambda$, $\lambda \in \Lambda_{K_{i+1}} \backslash \Lambda_{K_i}$, are then matched with the sign of $\widehat{g}_\lambda$, $\lambda \in \Lambda_{K_i}$, repeating the same procedure as before.

For any $x \in \mathbb{R}^d$, this procedure will eventually involve a $U_\nu$ for which $x\in U_\nu$, 
by the connectedness of 
$\mathbb{R}^d$, since the $U_\lambda$ form a covering of $\mathbb{R}^d$ from open sets. 
It follows that this determines the whole family $\{f_\lambda\}_{\lambda \in \Lambda}$ up to one global sign factor, and hence $f$ 
itself up to one global sign factor.

If $h \in L^2(\mathbb{R}^d,\mathbb{R})$ satisfies  $|f \ast \psl|(X_\lambda) = |h \ast \psl|(X_\lambda)$ for all $\lambda \in \Lambda$, then, $h \ast \psi_\lambda = \sigma_\lambda f \ast \psi_\lambda$ for some $\sigma_\lambda \in \{-1,1\}$. Furthermore, we have $\widehat{f} \widehat{\psi}_\lambda = \sigma_\lambda \widehat{h} \widehat{\psi}_\lambda.$ Let $J_{\lambda, \mu}$ be defined as before so that $|\widehat{f}|>0$ on $J_{\lambda, \mu}$. Since in addition $|\widehat{\psi}_\lambda|>0$ on $\widetilde{F}_{\lambda,\mu}$, we also have $|\widehat{h}|>0$ on $J_{\lambda, \mu}$. Repeating the above argument then yields that $h=\sigma f$ for some $\sigma \in \{-1,1\}$.
\end{proof}
\begin{remark}
These conclusions also hold under more general conditions:
The requirement that $f \ast \psl$ be band-limited has been obtained by asking compact support for $\widehat{\psi}_\lambda$. Instead, one could also impose compactness of $\supp \widehat{f}$, so that weaker conditions on $\psl$ and adapted sampling rates would still give a similar solution. One would, however, need to impose additional conditions on $f$ to make up for the loss of analyticity of $\widehat{f}$, which played a role in the proof. The exponential decay of $f$ can also be slightly relaxed to include Fourier transforms of various quasi-analytic classes, e.g. $\int_{\mathbb{R}^d} |f(x)| e^{\alpha|x|/\log(1+|x|)}dx < \infty$, for which $J_{\lambda,\mu}$ as above can still be found.
\end{remark}

\section{Local stability}\label{sec:stability}
In this section, we study the stability of our reconstruction problem. Roughly speaking, we show that under the assumptions in Theorem \ref{mainprop}, the operator $T$ that maps a function $f$ to its samples $\{ |f \ast \psl |(X_\lambda)\}_{\lambda \in \Lambda}$ has closed range $\mathcal{R}(T)$ and is continuously invertible on $\mathcal{R}(T)$. 

Throughout this section we fix
\begin{itemize}
\item[--] a semi-discrete frame $\Psi_\Lambda$ of real-valued band-limited functions in $L^2(\mathbb{R}^d)$ with good F-support overlap (see Definition \ref{overlap-supp}) and with lower and upper frame bounds $A, B>0$, respectively; \\
\item[--] for each $\lambda \in \Lambda$, a sign-blind sampling set $X_\lambda$ for $\supp \widehat{\psi}_\lambda$ (see Definition \ref{approp-sampling}); \\
\item[--] a compact subset $C$ of $\mathbb{R}^d$.
\end{itemize}
In what follows, to ease on notation, we shall replace the sequence space $\ell^2(\Lambda,\ell^2(\mathbb{Z}^d))$ by its isomorphic cousin
$\ell^2(\Lambda \times \mathbb{Z}^d)$.
With these prerequisites we define the sampling operator as the mapping
\begin{align*}
T&: \, L^2(C,\mathbb{R})/\{-1,1\} \to \ell^2(\Lambda \times \mathbb{Z}^d), \\
T(f) &= \{ |F_\lambda|^{-\frac{1}{2}} |f \ast \psl |(X_\lambda)\}_{\lambda \in \Lambda},
\end{align*}
where $|F_\lambda|$ denotes the Lebesgue measure of $F_{\lambda}:=\supp \widehat{\psi}_\lambda$, $$\| f-g \|_{L^2(C,\mathbb{R})/\{-1,1\}}:= \min\{ \| f-g\|_{L^2(C,\mathbb{R})L^2(C,\mathbb{R})}, \| f+g\|_{L^2(C,\mathbb{R})} \}$$ and $$\| \{ \{ h_{\lambda,\bn}\}_{\bn \in \mathbb{Z}^d} \}_{\lambda \in \Lambda} \|_{\ell^2(\Lambda \times \mathbb{Z}^d)}:= \sum_{\lambda \in \Lambda, \bn \in \mathbb{Z}^d} |h_{\lambda,\bn}|^2.$$
\begin{remark}\leavevmode
\begin{enumerate}
\item The normalization by $|F_\lambda|^{-1/2}$ is introduced to ensure the convergence of the sums in $\ell^2(\Lambda, \ell^2(\mathbb{Z}^d))$.\\
\item For the proof of Theorem \ref{thm-stab} below, we need to consider a function class that is a closed set. Therefore we defined $T$ to act on functions in $L^2(C,\mathbb{R})$ for $C$ a fixed compact set.
\end{enumerate}
\end{remark}
We are now ready to state the main result of this section.
\begin{theorem}\label{thm-stab}
The operator $T$ is injective and its range $\mathcal{R}(T):=T(L^2(C,\mathbb{R})/\{-1,1\})$ is closed in $\ell^2(\Lambda \times \mathbb{Z}^d)$. The inverse map $$T^{-1}: \mathcal{R}(T) \subseteq \ell^2(\Lambda \times \mathbb{Z}^d) \to L^2(C,\mathbb{R})/\{-1,1\}$$ is continuous.
\end{theorem}
\begin{proof}
We need to show that for any Cauchy sequence $\{h^m \}_{m \in \mathbb{N}} \subset \mathcal{R}(T)$ with $h^m \to h$ in $\ell^2(\Lambda\times\mathbb{Z}^d)$, $\{ T^{-1}(h^m) \}_{m \in \mathbb{N}}$ is a Cauchy sequence converging to $u$ and $T(u) = h$, i.e., $h \in \mathcal{R}(T)$. Our proof follows the lines of \cite[Section 3.2]{mw}.

Let $\{ h^m \}_{m \in \mathbb{N}} :=\{ \{ |F_\lambda|^{-1/2} |f^m \ast \psl|(X_\lambda)\}_{\lambda \in \Lambda} \}_{m \in \mathbb{N}}  \subset \mathcal{R}(T)$ be a Cauchy sequence in $\ell^2(\Lambda \times \mathbb{Z}^d)$. A theorem by Fr\'echet \cite{wloka1971} states that a subset $\{x^m \}_{m \in \mathbb{N}} \subset \ell^2(\mathbb{N})$ is relatively compact if and only if
\begin{itemize}
\item $\sup_{m \in \mathbb{N}} \|x^m \|_{\ell^2(\mathbb{N})} < \infty$ and \\
\item $\sup_{m \in \mathbb{N}} \sum_{k>N} |x_k^m|^2 \to 0$ as $N \to \infty$.
\end{itemize}
A straightforward corollary of this theorem is that if $\{|x^m| \}_{m \in \mathbb{N}} \subset \ell^2(\mathbb{N})$ is relatively compact, then $\{x^m \}_{m \in \mathbb{N}} \subset \ell^2(\mathbb{N})$ is also relatively compact.

Clearly, the sequence $\{h^m\}_{m \in \mathbb{N}}$ is relatively compact in $\ell^2(\Lambda \times \mathbb{Z}^d)$ because it is a Cauchy sequence. Identifying $\Lambda \times \mathbb{Z}^d$ with $\mathbb{N}$, we obtain that $\{ \{ |F_\lambda|^{-1/2} f^m \ast \psl(X_\lambda)\}_{\lambda \in \Lambda} \}_{m \in \mathbb{N}} \subset \ell^2(\Lambda \times \mathbb{Z}^d)$ is relatively compact. Therefore, it has a subsequence  $\{ \{ |F_\lambda|^{-1/2} f^{m_j} \ast \psl(X_\lambda)\}_{\lambda \in \Lambda} \}_{j \in \mathbb{N}}$ that converges in $\ell^2(\Lambda \times \mathbb{Z}^d)$.

Next, we need to show that the above implies that $$\{ \{ f^{m_j} \ast \psl \}_{\lambda \in \Lambda} \}_{j \in \mathbb{N}} \subset \ell^2(\Lambda, L^2(\mathbb{R}^d,\mathbb{R}))$$ is a Cauchy sequence. Since each $X_\lambda$ is a sign-blind sampling set for $\supp \widehat{\psi}_\lambda$, it is also a set of stable sampling for $\supp \widehat{\psi}_\lambda$, so that standard results in sampling theory imply that reconstructing a function $ g \ast \psl$ from its samples $|F_\lambda|^{-1/2} g \ast \psl (X_\lambda)$ is stable from $L^2(\mathbb{R}^d,\mathbb{R})$ to $\ell^2(\mathbb{Z}^d)$. More precisely, due to the normalization by $|F_\lambda|^{-1/2}$, there exists a constant $c>0$ such that for all $\lambda \in \Lambda$
\begin{equation}
\|g \ast \psl \|_{L^2(\mathbb{R}^d, \mathbb{R})} \leq c \||F_\lambda|^{-1/2} g \ast \psl (X_\lambda) \|_{\ell^2(\mathbb{Z}^d)}.
\end{equation}
This implies that for all $j,k \in \mathbb{N}$
\begin{align*}
\sum_{\lambda \in \Lambda} & \| f^{m_j} \ast \psl - f^{m_{j+k}} \ast \psl \|_{L^2(\mathbb{R}^d, \mathbb{R})}^2 \leq \\
& \leq c^2 \sum_{\lambda \in \Lambda} \| |F_\lambda|^{-1/2} f^{m_j} \ast \psl(X_\lambda)  - |F_\lambda|^{-1/2} f^{m_{j+k}} \ast \psl(X_\lambda) \|_{\ell^2(\mathbb{Z}^d)}^2.
\end{align*}
Since $\{ \{ |F_\lambda|^{-1/2} f^{m_j} \ast \psl(X_\lambda)\}_{\lambda \in \Lambda} \}_{j \in \mathbb{N}}$ is a Cauchy sequence, this results in $\{ \{ f^{m_j} \ast \psl \}_{\lambda \in \Lambda} \}_{j \in \mathbb{N}}$ being a Cauchy sequence in $\ell^2(\Lambda, L^2(\mathbb{R}^d,\mathbb{R}))$. We denote its limit by $\{ g_\lambda \}_{\lambda \in \Lambda}$.

As a next step we prove the existence of a function $u \in L^2(\mathbb{R}^d,\mathbb{R})$ with $g_\lambda = u \ast \psl$ for all $\lambda \in \Lambda$. To see this we define $\hat{u}\restriction_{F_\lambda} :=  \widehat{g_\lambda}/\widehat{\psl}$ and show that this definition is consistent, i.e., that
$$\frac{\widehat{g_\lambda}}{\widehat{\psl}} = \frac{\widehat{g_\mu}}{\widehat{\psi_\mu}} \text{ on } F_\lambda \cap F_\mu.$$
Indeed: by the definition of $g_\lambda$, this is equivalent to
$$\lim_{j \to \infty} \widehat{f^{m_j}} \widehat{\psi_\lambda} \widehat{\psi_\mu} = \lim_{j \to \infty} \widehat{f^{m_j}} \widehat{\psi_\mu} \widehat{\psi_\lambda},$$
which is obviously true. It follows that $\hat{u}$ is defined consistently on all of $\mathbb{R}$. We next prove that $f^{m_j} \to u$ in $L^2(\mathbb{R}^d, \mathbb{R})$ by the following argument:
\begin{align*}
\|f^{m_j} - u\|_{L^2(\mathbb{R}^d, \mathbb{R})}^2 \leq A^{-1} \sum_{\lambda \in \Lambda} \| f^{m_j} \ast \psl - u \ast \psl \|_{L^2(\mathbb{R}^d,\mathbb{R})}^2 \to 0 \text{ as } j \to \infty,
\end{align*}
where we have used that $\{u \ast \psl \}_{\lambda \in \Lambda}
=\{g_\lambda\}_{\lambda \in \Lambda} $ is the limit of the Cauchy sequence $\{ \{ f^{m_j} \ast \psl \}_{\lambda \in \Lambda} \}_{j \in \mathbb{N}}$. By the continuity of $T$ this immediately implies $T(u)=h$. Note that since $ f^{m_j} \in L^2(C,\mathbb{R})$, we also have that $u \in L^2(C,\mathbb{R})$. 
In fact, if $\tilde{u} \in L^2(C,\mathbb{R})$ is any accumulation point of the sequence $f^m$, repeating the above argument would lead to $T(\tilde{u})=h$. Since by Theorem \ref{mainprop} the operator $T$ is injective, this leads to $\tilde{u}=u$. Thus, the 
sequence $\{T^{-1}(h^m) \}_{m \in \mathbb{N}}$ has $u$ as its unique accumulation point -- in other words, it is a Cauchy sequence converging to $u$. 
\end{proof}

Because we have no quantitative estimates for the stability, this theorem 
is of at most theoretical interest. The local stability result in Theorem \ref{thm-stab} asserts that the nonlinear mapping $T^{-1}$ is continuous, but
this does not imply any uniform continuity for $T^{-1}$. In fact, the one-dimensional phase retrieval problem studied in \cite{mw} is 
explicitly shown there to {\em not} be uniformly continuous.

It turns out that the problem of reconstructing a real-valued band-limited function $f$ (in our case these are the $g_\lambda$) from its unsigned samples is {\em always} unstable. 
This is a special case of the more general recent result \cite{cahill-daub} that phase retrieval is always unstable in the infinite dimensional setting: when the signal space is an infinite dimensional Hilbert space $\mathcal{H}$, then no matter how one chooses the discrete frame $\{\varphi_\gamma \}_{\gamma\in \Gamma}$, one can always find, for any $\epsilon > 0$, elements $f_{\epsilon},g_{\epsilon} \in \mathcal{H}$ such that $\left\|f_{\epsilon}\right\|_\mathcal{H}, \left\|g_{\epsilon}\right\|_\mathcal{H} \leq 1$, $\sum_{\gamma \in \Gamma} \left( \left|\langle f_{\epsilon}, \varphi_\gamma\rangle\right| - \left|\langle g_{\epsilon}, \varphi_\gamma\rangle\right|\right)^2 < \epsilon$, yet $\left\|f_{\epsilon} -g_{\epsilon}\right\|_\mathcal{H}\geq c$, where $c>0$ is some fixed constant independent of $\epsilon$ (but depending on the frame). If the
frame allows for unique (but unstable) phase retrieval in 
$\mathcal{H}$, then for signals $f$ restricted to finite-dimensional subspaces of $\mathcal{H}$, stable reconstruction is possible (see
e.g. \cite{candes2015phase}), but
the stability estimate depends on a factor that scales with the dimension  of the subspace, diverging to infinity as this dimension grows 
unboundedly \cite{cahill-daub}. 

In the algorithm considered in \cite{thakur}, reconstruction for 
real-valued and band-limited $g \in L^2(\mathbb{R})$ is based on choosing a line $L = \{ z: \text{Im }z = c  \}$ for some constant $c$ on which to unwrap the phase. Since $g$ is analytic on $\mathbb{C}$, $g$ has no zeros on almost every such line $L$. The instability of the reconstruction then
manifests itself in the possibility for $g$ to have zeros arbitrarily close to $L$. In fact, one can construct band-limited functions such that for any $c \in \mathbb{R}$, and any (arbitrarily narrow) strip $S_L$ around $L$, $g$ has zeros in $S_L$, see \cite{thakur2011three}.

Note that increasing the sampling rate does not enhance the stability of the reconstruction. In view of the possibility that $g$ has zeros close to the line $L$, it seems intuitive that the instability cannot be overcome by a higher sampling rate. This observation is also consistent with the discussion on higher sampling rates in \cite{candes2015phase}.

\section{Examples}\label{sec:examples}

This section contains two examples illustrating the concepts discussed so far. We will need the notion of a (discrete) frame and a tight frame of $L^2(\mathbb{R}^d)$, which we now introduce. Let $\mm$ be a discrete index set and $\Omega=\{\omega_\mu\}_{\mu \in \mm} \subset L^2(\mathbb{R}^d)$. If there exist two positive constants $A \leq B < \infty$ such that for all $f \in L^2(\mathbb{R}^d)$
$$A\|f\|^2 \leq \sum_{\mu \in \mm} | \langle f, \omega_\mu \rangle|^2 \leq B \|f\|^2$$
holds, then $\Omega$ is a frame of $L^2(\mathbb{R}^d)$. If, in addition, $A=B$, then $\Omega$ is called a \textit{tight} frame. 

We borrow a phrasing from \cite{cahill-daub} and say that a frame $\Omega=\{\omega_\mu\}_{\mu \in \mm} \subset L^2(\mathbb{R}^d)$ of real-valued functions $\omega_\mu$ \textit{does sign retrieval} if any 
exponentially-decaying real-valued function $f \in L^2(\mathbb{R}^d)$ can be uniquely recovered (up to a global sign factor) from its unsigned frame coefficients $\{ |\langle f, \omega_\mu \rangle |\}_{\mu \in \mm}$. In what follows we will show that for the orthonormal basis of Meyer wavelets in 1D and the tight frame of second generation curvelets in 2D, suitably oversampled frames do sign retrieval. 

\subsection*{1D: Meyer wavelets} In $L^2(\mathbb{R})$, we consider an orthonormal basis $$\Phi=\{ \psi_{j,k}\}_{j \in \mathbb{N}, k \in \mathbb{Z}} \cup \{ \phi_{0,k}\}_{k \in \mathbb{Z}}$$ of Meyer wavelets $ \psi_{j,k}(\cdot) = 2^{j/2} \psi (2^j \cdot - k)$ and scaling functions $\phi_{0,k}(\cdot) = \phi(\cdot-k)$, all of which have compactly supported Fourier transforms, \cite{meyer1995wavelets}. In the frequency domain, the mother wavelet $\psi$ and the scaling function $\phi$ can be defined by
\begin{align*}
\widehat{\psi}(\xi) =
  \begin{cases} 
      \hfill \sin\big[\frac{\pi}{2} \beta(3|\xi|-1) \big] e^{i \pi \xi}   \hfill & |\xi| \in [\frac{1}{3},\frac{2}{3}] \\
      \hfill \cos\big[\frac{\pi}{2} \beta(3|\xi|/2-1) \big] e^{i \pi \xi} \hfill & |\xi| \in [\frac{2}{3},\frac{4}{3}] \\
      \hfill 0 \hfill & \text{otherwise}
  \end{cases}
\end{align*}
and 
\begin{align*}
\widehat{\phi}(\xi) =
  \begin{cases} 
      \hfill \cos\big[\frac{\pi}{2} \beta(3|\xi|-1) \big]    \hfill & |\xi| \leq 2/3 \\
      \hfill 0 \hfill & |\xi| > 2/3. \\
  \end{cases}
\end{align*}
Here, $\beta(x)$ can be any function going (smoothly) from $0$ to $1$ on $[0,1]$ that satisfies
$$\beta(x) + \beta(1-x)=1, \ \text{for all } x \in [0,1].$$
For $\alpha <1$, we further define the family of functions 
$$\Phi^\alpha = \{ \psi_{j,k}^\alpha(\cdot) := 2^{j/2} \psi(2^j \cdot - \alpha k)\}_{j \in \mathbb{N}, k \in \mathbb{Z}} \cup \{ \phi_{0,k}^\alpha(\cdot) = \phi(\cdot - \alpha k) \}_{k \in \mathbb{Z}}.$$
It is not hard to show that 
$\Phi^\alpha$ constitutes a frame for any $\alpha < 1$.

\begin{theorem}\label{phr-meyer}
If $\alpha \le 3/16$, then $\Phi^{\alpha}$ does sign retrieval.
\end{theorem}
\begin{proof}
Let $\Psi$ be the semi-discrete frame formed by 
$$\Psi = \{\psi_j(\cdot) := 2^j \psi(2^j \cdot)\}_{j \in \mathbb{N}} \cup \{ \phi \}.$$
We denote the supports in frequency domain by
\begin{align*}
F_{-1}:= \supp \widehat{\phi} &= \Big[-\frac{2}{3}, \frac{2}{3} \Big]
\end{align*}
and 
\begin{align*}
F_j:=\supp \widehat{\psi}_{j} &= \Big[- \frac{2^{j+2}}{3}, - \frac{2^{j}}{3}\Big] \cup \Big[ \frac{2^{j}}{3},  \frac{2^{j+2}}{3}\Big], \ j \in \mathbb{N}. 
\end{align*}
Clearly, this semi-discrete frame has good F-support overlap. Furthermore, for all $j \in \mathbb{N}$ and for any $c_j \geq 2^{j+3}/3,$ the set $X_j= \frac{1}{2c_j} \mathbb{Z}$ is a sign-blind sampling set for $F_j$. By Theorem \ref{mainprop}, any exponentially decaying real-valued function $f \in L^2(\mathbb{R})$ can be uniquely recovered from its unsigned samples $\{|f \ast \psi_j|(X_j)\}_{j \in \mathbb{N}} \cup \{|f \ast \phi|(X_{-1})\}$ up to a global sign factor. Note, however, that
\begin{align*}
|f \ast \psi_j|(X_j) &= 2^{j/2}|\langle f, \psi_{j,k}^\alpha \rangle|_{j \in \mathbb{N}, k \in \mathbb{Z}},\\
|f \ast \phi|(X_{-1}) &= |\langle f, \phi_{0,k}^\alpha \rangle|_{k \in \mathbb{Z}},
\end{align*}
for $\alpha=2^{j-1}/c_j.$ Therefore, $\Phi^{\alpha}=\{ \psi_{j,k}^\alpha\}_{j \in \mathbb{N}, k \in \mathbb{Z}} \cup \{ \phi_{0,k}^\alpha\}_{k \in \mathbb{Z}}$ does sign retrieval for any $c_j\geq 2^{j+3}/3$, which is equivalent to $\alpha \le 3/16$.

\end{proof}
\subsection*{2D: Curvelets}
We follow the construction in \cite{candes2004new} of \textit{second generation} curvelets. For this, let $\nu$ be an even $C^\infty$-function with support on $[-1/2,1/2]$ satisfying
$$|\nu^2(\theta)|^2+|\nu(\theta-1/2)|^2=1, \ \theta \in [0,1)$$
and $w$ a $C^{\infty}$-function supported on $[1/3,8/3]$ which is obtained from a construction of Meyer wavelets (see \cite{candes2004new} for details).  For $j \geq 1$ and $\ell \in \Theta=\{ 0, 1, 2, \dots, 2^j-1\}$, we consider the angular window $\nu_{j,\ell}(\theta) = \nu(2^j \theta - \ell/2)$. Using this, one can define the windows $$\chi_{j,\ell}(\xi) = w(2^{-2j}|\xi|) (\nu_{j,\ell}(\theta)+\nu_{j,\ell}(\theta+1/2));$$
together with some function $\chi_0(\xi)$ with support in a disk centered at the origin (and hence capturing the low-frequency components), these yield a tiling of the frequency domain.

Let $\gamma_j$ be the inverse Fourier transform of $2^{-3j/2-5/2} \chi_{j,0}(\xi)$. In \cite{candes2004new} it is shown that the system of curvelets $\Gamma= \{ \gamma_{j,\ell,k} \}_{j \in \mathbb{N},\ell \in \Theta,k \in \mathbb{Z}^2}$ constructed as
$$\gamma_{j,0,k}(x)=\gamma_j(x_1-2^{-2j}k_1/\delta_1,x_2-2^{-j}k_2/\delta_2)$$
with $\delta_1=14/3(1+O(2^{-j}))$, $\delta_2=10\pi/9$ and $\gamma_{j,\ell,k}(x)=\gamma_{j,0,k}(R_{\theta_{j,\ell}}^*)$ constitutes a tight frame. Here, $R_{\theta_{j,\ell}}$ denotes the rotation by $\theta_{j,\ell}=2^{-j}\pi \ell$.

We now consider oversampled frames of the form $\Gamma^\alpha= \{ \gamma^\alpha_{j,\ell,k} \}_{j \in \mathbb{N},\ell \in \Theta,k \in \mathbb{Z}^2}$, where for $\alpha=(\alpha_1,\alpha_2)$,
$$\gamma^\alpha_{j,0,k}(x)=\gamma_j(x_1-2^{-2j}\alpha_1 k_1,x_2-2^{-j} \alpha_2 k_2)$$
and $\gamma^\alpha_{j,\ell,k}(x)=\gamma^\alpha_{j,0,k}(R_{\theta_{j,\ell}}^*x)$. 

Our goal is to identify values of $\alpha$ such that $\Gamma^\alpha$ does sign retrieval. 
For each $j\geq 1$, $F_{j,0}:=\supp \chi_{j,0}$ is contained in the open rectangle $(-c_{j,1}/2,c_{j,1}/2) \times (-c_{j,2}/2,c_{j,2}/2)$ for any constants $c_{j,1}\geq\frac{1}{3}2^{2j+4}$, $c_{j,2}\geq \frac{20\pi}{9}2^{j-1}$. 
Thus, $X_j= D[(3 \cdot 2^{-2j-5},\frac{9}{20\pi}2^{-j})] \mathbb{Z}^2$ yields a sign-blind sampling set for $F_{j,0}$. By setting $\alpha_1=2^{j-1}/c_{j,1}$, $\alpha_2 = 2^{j-1}/c_{j,2}$, one can -- similarly to the previous example with Meyer wavelets -- deduce the following result:

\begin{theorem}\label{phr-curvelets}
If $\alpha=(\alpha_1,\alpha_2)$ is chosen such that $\alpha_1\le 3 \cdot 2^{-5}$ and $\alpha_2 \le \frac{9}{20 \pi}$, then $\Gamma^\alpha$ does sign retrieval.
\end{theorem}
 
Compared to $\Gamma$ the additional redundancy in $\Gamma^\alpha$ is of a factor $\alpha_1^{-1}/\delta_1\approx 2.29$ in the first component and of a factor $\alpha_2^{-1}/\delta_2 \approx 2$ in the second component, resulting overall in a redundancy factor $\frac{\alpha_1^{-1}}{\delta_1} \cdot \frac{\alpha_2^{-1}}{\delta_2} \le 4.58$.

It is easy to see that also for other frame constructions that are based on a dyadic frequency decomposition \cite{grohs2014alpha} (e.g. shearlets, ridgelets,...), one can oversample with a fixed scale-independent rate to obtain a frame that does sign retrieval. Finally, we remark that such results do not necessarily hold if the construction of the frame is not based on a dyadic decomposition of the frequency domain, as is the case for example for wave atoms \cite{demanet2007wave} or Gabor frames \cite{grochenig-book}. There, the oversampling factors will be scale dependent.

\bigskip
\textbf{Acknowledgments. }{R.A. is supported by an ETH Postdoctoral Fellowship. R.A., I.D. and P.G. would like to thank the Mathematisches Forschungsinstitut Oberwolfach (MFO). The authors also give their thanks to the anonymous referee for the detailed and useful comments and for inspiring our example given after Proposition \ref{thakur1-multid-simple}.}
\bibliographystyle{plain}
\bibliography{phase-retrieval}

\end{document}